\documentclass[preprint,12pt]{elsarticle}

\everymath{\displaystyle}
\newcommand{\comment}[1]{}
\usepackage{amssymb}
\usepackage{amsthm}
\usepackage{amsmath}
\usepackage{amsfonts}
\newtheorem{theorem}{Theorem}[section]
\newtheorem{lemma}[theorem]{Lemma}
\newtheorem{corollary}[theorem]{Corollary}

\journal{Discrete Mathematics}

\begin{document}

\begin{frontmatter}

\title{Determinants of box products of paths}
\author{Daniel Pragel}
\address{Emory and Henry College\\ P.O. Box 947\\ Emory, VA 24327}

\begin{abstract}
Suppose that $G$ is the graph obtained by taking the box product of a path of length $n$ and a path of length $m.$ Let $\mathbf{M}$ be the adjacency matrix of $G.$ If $n = m,$ H.M. Rara showed in 1996 that $\det(\mathbf{M}) = 0.$ We extend this result to allow $n$ and $m$ to be any positive integers, and show that $$\det(\mathbf{M}) = \left\{\begin{array}{ll}0 & \hbox{if } \gcd(n+1,m+1) \neq 1,\\ (-1)^{nm/2} & \hbox{if } \gcd(n+1,m+1) = 1.\end{array}\right.$$

\end{abstract}

\begin{keyword} Graph Theory \sep Box Product \sep Cartesian Product \sep Adjacency Matrix \sep Path \sep Determinant

\end{keyword}

\end{frontmatter}

\section{Introduction}
Let $[n] = \{1,2,\dots,n\}.$ We define a \textit{graph} $G$ to be an ordered pair of sets $(V,E),$ where $V$ is any set and $E \subseteq \binom{V}{2}$; we refer to $V$ as the \textit{vertices} and $E$ as the \textit{edges} of $G.$ The \textit{adjacency matrix} of $G$ is denoted $\mathbf{A}(G)$ and is a matrix with rows and columns indexed by $V$ such that $$\mathbf{A}(G)_{i,j} = \left\{\begin{array}{l l} 1 & \hbox{if } \{i,j\} \in E,\\ 0 & \hbox{if } \{i,j\} \not\in E.\end{array}\right.$$  Let $\mathbf{I}_n$ be the $n \times n$ identity matrix and let $\mathbf{0} _n$ be the $n \times n$ matrix of all zeros. If $G$ has $n$ vertices, the \textit{characteristic polynomial} of $\mathbf{A}(G)$ is defined to be $q_G(x) = \det(\mathbf{A}(G) - x\mathbf{I}_n).$

Suppose $G_1$ and $G_2$ are graphs with vertex sets $V_1$ and $V_2,$ and edge sets $E_1$ and $E_2,$ respectively. The \textit{box product} of $G_1$ and $G_2,$ denoted $G_1 \Box G_2,$ is the graph with vertex set $V = V_1 \times V_2$ and such that, for $i_1,j_1 \in V_1$ and $i_2,j_2 \in V_2,$ $\{(i_1,i_2),(j_1,j_2)\}$ is an edge in $G_1 \Box G_2$ if and only if either $i_1 = j_1$ and $\{i_2,j_2\} \in E_2,$ or $i_2 = j_2$ and $\{i_1,j_1\} \in E_2.$ For an in-depth look at the box product (also referred to as the \textit{Cartesian product}) of graphs, see \cite{imrich}.

Let $G$ be a graph with vertex set $[n]$ and adjacency matrix $\mathbf{A},$ and let $H$ be a graph with vertex set $[m]$ and adjacency matrix $\mathbf{B}.$ Then, the vertices of $G \Box H$ can be labeled with the elements of $[nm],$ by relabeling the vertex $(i,j)$ as $(i-1)m + j.$ Under this labeling, the adjacency matrix $\mathbf{M}$ of $G \Box H$ can be written as an $n \times n$ block matrix $\mathbf{M} = [\mathbf{M}_{i,j}],$ where each $\mathbf{M}_{i,j}$ is $m \times m.$ Further, $$\mathbf{M}_{i,j} = \left\{\begin{array}{l l} \mathbf{B} & \hbox{if $i = j,$}\\ \mathbf{I}_m & \hbox{if $i \neq j$ and $i \sim j$ in $G$,}\\ \mathbf{0} _m & \hbox{if $i \neq j$ and $i \not\sim j$ in $G$.}\end{array}\right.$$ The $\mathbf{M}_{i,j}$ are all elements of the commutative subring $S$ of $\mathbb{R}^{m \times m}$ generated by $\mathbf{B}$ and $\mathbf{I}_m.$ Thus, if we denote the determinant over the ring $S$ by ${\det}_{S},$ it is not hard to see that ${\det}_S(\mathbf{M}) = q_G(-\mathbf{B}),$ so $${\det}(\mathbf{M}) = {\det}\left({\det}_S(\mathbf{M})\right) = {\det}(q_G(-\mathbf{B})).$$ We now consider the case when both $G$ and $H$ are \textit{paths}.
\section{Paths and Products of Paths}
The \textit{path with $n$ vertices,} denoted $P_n,$ is the graph with vertex set $V = [n]$ and edge set $E = \{(i,i+1): i \in [n-1]\}.$ Let $q_n(x)$ be the characteristic polynomial of $\mathbf{A}(P_n).$ In \cite{rara}, it was shown that $\det\left(\mathbf{A}\left(P_n \Box P_n\right)\right) = 0.$ We extend this result, and compute the value of $\det\left(\mathbf{A}\left(P_n \Box P_m\right)\right)$ for all positive integers $n$ and $m.$ We do this first by looking at $q_n(x).$ Note that, since $\mathbf{A}\left(P_n\right)$ is a tridiagonal matrix and has a very simple structure, many of the properties, including the roots, of $q_n(x)$ are explicitly known; for example, see \cite{fonseca1} and \cite{fonseca2}. We will take advantage of a few particularly nice properties of $q_n(x).$ First, we will use the following theorem from \cite{schwenk}. We add our own corollary below.

\begin{theorem}\label{schwenk} For $n \geq 2,$ $q_n(x) = -xq_{n-1}(x) - q_{n-2}(x).$
\end{theorem} \qed

\begin{corollary}\label{parity} Let $n \geq 0.$ If $n$ is even, $q_n(x)$ is an even polynomial. If $n$ is odd, $q_n(x)$ is an odd polynomial.
\end{corollary}

\begin{proof} By inspection, corollary \ref{parity} is true for $n \leq 2.$ Assume it is true for all $n' < n$ for some $n > 2.$ This implies that $q_{n-1}(x)$ and $q_{n-2}(x)$ have opposite parities as polynomials, so $xq_{n-1}(x)$ and $q_{n-2}(x)$ have the same parity, so, using theorem \ref{schwenk}, we see that $q_n(x)$ and $q_{n-2}(x)$ have the same parity. The result follows.
\end{proof}

We will also use the following lemma; for a proof, see \cite{shi}.

\begin{lemma}\label{shi}For any $k \geq 1,$ if $i \in [k-1],$ then $$q_k(x) = q_i(x)q_{k - i}(x) - q_{i-1}(x)q_{k - i-1}(x).$$ Further, if $q_k(\lambda) = 0,$ then the following statements are true as well.

\begin{itemize}
\item[\textnormal{(a)}] If $0 \leq s \leq k,$ then $q_{k + s}(\lambda) = - q_{k - s}(\lambda).$
\item[\textnormal{(b)}] If $t \geq 1,$ then $q_{t(k + 1)-1}(\lambda) = 0.$
\end{itemize}
 \end{lemma} \qed

We now are ready to prove the below theorem.

\begin{theorem}\label{recur} Suppose that $q_k(\lambda) = 0.$ Then, for all $a \geq 1$ and $0 \leq b \leq k,$
$q_{a(k+1) + b}(\lambda) = \left(q_{k+1}(\lambda)\right)^aq_b(\lambda).$
\end{theorem}

\begin{proof} Note that theorem \ref{recur} trivially holds when $a = 1$ and $b = 0.$ Suppose $1 \leq b \leq k.$ Applying lemma \ref{shi} shows that $$q_{k + 1 + b} = q_{k+1}(\lambda)q_{b}(\lambda) - q_{k}(\lambda)q_{b-1}(\lambda) = q_{k+1}(\lambda)q_b(\lambda),$$ and thus theorem \ref{recur} holds when $a = 1$ and $0 \leq b \leq k.$ Suppose it holds when $1\leq a < a'$ and $0 \leq b \leq k,$ for some $a' > 1.$ Suppose that $0 \leq b \leq k.$ Then, by lemma \ref{shi}, \begin{align*}q_{a'(k+1) + b}(\lambda) &= q_{(a'-1)(k+1) + b + k+1}(\lambda)\\
&= q_{(a'-1)(k+1) + b}(\lambda)q_{k+1}(\lambda) + q_{(a'-1)(k+1)+b-1}(\lambda)q_{k}(\lambda)\\
&= q_{(a'-1)(k+1) + b}(\lambda)q_{k+1}(\lambda) = \left(q_{k+1}(\l)\right)^{a'-1}q_b(\lambda)q_{k+1}(\lambda)\\ &= \left(q_{k+1}(\lambda)\right)^{a'}q_b(\lambda).\end{align*}
\end{proof}

Label the roots of $q_n(x)$ as $\lambda_{n,1},\lambda_{n,2}, \cdots,\lambda_{n,n}.$ Using our result from the previous section, $\det(\mathbf{A}(P_n \Box P_m)) = \det(q_n(-\mathbf{A}(P_m))).$ Corollary \ref{parity} implies that $q_n(-\mathbf{A}(P_m)) = (-1)^n q_n(\mathbf{A}(P_m)).$ Further, we can factor $q_n(x)$ as $$q_n(x) = \prod_{i = 1}^n \left(\lambda_{n,i} - x\right) = (-1)^n \prod_{i = 1}^n\left(x - \lambda_{n,i}\right).$$ Thus, \begin{align*}\det(\mathbf{A}(P_n \Box P_m)) &= \det(q_n(-\mathbf{A}(P_m))) = \det\left((-1)^n q_n(\mathbf{A}(P_m))\right)\\ &= \det\left((-1)^n(-1)^n \prod_{i = 1}^n\left(\mathbf{A}(P_m) - \lambda_{n,i}\mathbf{I}_m \right)\right)\\ &= \det\left(\prod_{i=1}^n (\mathbf{A}(P_m) - \lambda_{n,i}\mathbf{I}_m)\right)\\ &= \prod_{i = 1}^n \det(\mathbf{A}(P_m) - \lambda_{n,i}\mathbf{I}_m) = \prod_{i = 1}^n q_m(\lambda_{n,i}).\end{align*} Since, by definition, it is immediately evident that $P_n \Box P_m$ and $P_m \Box P_n$ are isomorphic as graphs, it follows that $\det(\mathbf{A}(P_n \Box P_m)) = \det(\mathbf{A}(P_m \Box P_n)).$ Thus, $$\prod_{i = 1}^n q_m(\lambda_{n,i}) = \prod_{i = 1}^m q_n(\lambda_{m,i}).$$ This leads to the following results.

\begin{theorem}\label{m=n+1} Suppose that $n \geq 1.$ Then, $\prod_{i = 1}^n q_{n+1}(\lambda_{n,i}) = (-1)^{n(n+1)/2}.$
\end{theorem}

\begin{proof} By inspection, theorem \ref{m=n+1} is true for $n = 1.$ Suppose it is true for some $n \geq 1.$ Then, by lemma \ref{shi}, $q_{n+2}(\lambda_{n+1,i}) = - q_n(\lambda_{n+1,i})$ for any $i \in [n+1],$ so \begin{align*}\prod_{i = 1}^{n+1} q_{n+2}(\lambda_{n+1,i}) &= \prod_{i=1}^{n+1} -q_n(\lambda_{n+1,i}) = (-1)^{n+1}\prod_{i=1}^{n+1} q_n(\lambda_{n+1,i})\\ &= (-1)^{n+1}\prod_{i=1}^n q_{n+1}(\lambda_{n,i}) = (-1)^{n+1}(-1)^{n(n+1)/2}\\ &= (-1)^{n+1+n(n+1)/2} = (-1)^{(n+1)(1+n/2)} = (-1)^{(n+1)(n+2)/2}.\end{align*}
\end{proof}

\begin{theorem}\label{Irule} Suppose $n, m \geq 1.$ Then, $$\prod_{i = 1}^n q_m (\lambda_{n,i}) = \left\{\begin{array}{l l} 0 & \hbox{if $\gcd(n+1,m+1) \neq 1$},\\ (-1)^{nm/2} & \hbox{if $\gcd(n+1,m+1) = 1.$}\end{array}\right.$$
\end{theorem}
\begin{proof} Note that the above product is the determinant of $\mathbf{A}(P_n \Box P_m),$ which, as discussed above, is equal to the determinant of $\mathbf{A}(P_m \Box P_n).$ Thus, without loss of generality, we may assume that $n \leq m.$ We will induct on the remainder when $m + 1$ is divided by $n+1.$ Suppose this remainder is 0. Then, $\gcd(n+1,m+1) \neq 1,$ and $m+1 = k(n+1)$ for some $k\geq 1,$ so $m = k(n + 1) - 1.$ Thus, by lemma \ref{shi}, $q_m(\lambda_{n,i}) = 0$ for $i \in [n],$ and it follows that the product of these terms is zero. This verifies theorem \ref{Irule} for this case.

Suppose the remainder when $m+1$ is divided by $n+1$ is 1; we then have $m + 1 = k(n+1) + 1$ for some $k \geq 1.$ Note that this implies that $\gcd(n+1,m+1) = 1$ and $m = k(n+1),$ so, by theorem \ref{recur}, for $i \in [n],$ $$q_m(\lambda_{n,i}) = q_{k(n+1)}(\lambda_{n,i}) = (q_{n+1}(\lambda_{n,i}))^kq_0(\lambda_{n,i}) = (q_{n+1}(\lambda_{n,i}))^k.$$ Thus, $$\prod_{i=1}^nq_m(\lambda_{n,i}) = \prod_{i=1}^n(q_{n+1}(\lambda_{n,i}))^k = \left(\prod_{i=1}^n q_{n+1}(\lambda_{n,i})\right)^k = \left((-1)^{n(n+1)/2}\right)^k,$$ by theorem \ref{m=n+1}. Further, $$\left((-1)^{n(n+1)/2}\right)^k = (-1)^{nk(n+1)/2} = (-1)^{nm/2}.$$ Thus, theorem \ref{Irule} is true in this case.

Finally, suppose theorem \ref{Irule} is true whenever the remainder when $(m+1)$ is divided by $(n+1)$ is less than $r,$ for some $r > 1.$ Then, consider any $(m+1)$ and $(n+1)$ with $(m+1)$ having remainder $r$ when divided by $(n+1).$ It follows that there exists $k \geq 1$ such that $m + 1 = k(n+1) + r,$ implying that $m = k(n+1) + r-1.$ Then, once again applying theorem \ref{recur}, \begin{align*}\prod_{i=1}^nq_{m}(\lambda_{n,i}) &= \prod_{i = 1}^n (q_{n+1}(\lambda_{n,i}))^k q_{r-1}(\lambda_{n,i}) = \prod_{i=1}^n(q_{n+1}(\lambda_{n,i}))^k\prod_{i=1}^nq_{r-1}(\lambda_{n,i})\\
&= \left(\prod_{i=1}^nq_{n+1}(\lambda_{n,i})\right)^k \prod_{i=1}^nq_{r-1}(\lambda_{n,i}) = (-1)^{nk(n+1)/2}\prod_{i=1}^n q_{r-1}(\lambda_{n,i})\end{align*} by above. Note that $\gcd(n+1,m+1)= \gcd(r,n+1),$ by construction. Further, the remainder when $(n+1)$ is divided by $r$ is less than $r.$ Thus, by our induction hypothesis, if $\gcd(n+1,m+1) \neq 1,$ then $\gcd(r,n+1) \neq 1,$ so, \begin{align*}\prod_{i=1}^nq_{m}(\lambda_{n,i}) &= (-1)^{nk(n+1)/2}\prod_{i=1}^n q_{r-1}(\lambda_{n,i}) = 0.\end{align*} Otherwise, $\gcd(n+1,m+1) = 1,$ so $\gcd(r,n+1) = 1,$ implying that \begin{align*}\prod_{i=1}^nq_{m}(\lambda_{n,i}) &= (-1)^{nk(n+1)/2}\prod_{i=1}^n q_{r-1}(\lambda_{n,i}) = (-1)^{nk(n+1)/2}(-1)^{n(r-1)/2}\\ &= (-1)^{n(k(n+1) + r-1)/2} = (-1)^{nm/2}.\end{align*}
\end{proof}
The following corollary to theorem \ref{Irule} follows immediately.

\begin{corollary} Suppose $n$ and $m$ are positive integers. Then, $$\det\left(\mathbf{A}(P_n \Box P_m)\right) = \left\{\begin{array}{l l} 0 & \hbox{if $\gcd(n+1,m+1) \neq 1$},\\ (-1)^{nm/2} & \hbox{if $\gcd(n+1,m+1) = 1.$}\end{array}\right.$$
\end{corollary} \qed

\end{document}